\documentclass{amsart}
\usepackage{graphicx} 

\usepackage{amsfonts,amsmath,amssymb,amsthm}

\usepackage{graphicx} 
\usepackage{float} 
\usepackage{subfigure} 

\newtheorem{The}{Theorem}[section]

\newtheorem{Cor}[The]{Corollary}
\newtheorem{Lem}[The]{Lemma}

\theoremstyle{definition}
\newtheorem{Def}[The]{Definition}
\newtheorem{Conj}[The]{Conjecture}
\newtheorem{Exa}[The]{Example}

\theoremstyle{remark}

\makeatletter
\newcommand{\rmnum}[1]{\romannumeral #1}
\newcommand{\Rmnum}[1]{\expandafter\@slowromancap\romannumeral #1@}
\makeatother

\title{Oda's conjecture for reflexive polytopes: some special cases}
\author{Binnan Tu}
\date{}

\begin{document}

\subjclass[2020]{Primary 52B20; Secondary 52B11, 52B12.}
\keywords{Oda's conjecture, IDP pair, simplicial reflexive polytopes, facet unimodular polytopes, almost co-unimodular pair}

\begin{abstract}
In this paper, we show that Oda's question holds for $n$-dimensional simplicial reflexive polytope $P$ and lattice polytope $Q$ containing the origin, when the vertex of $Q$ is either a vertex of $P$ or the origin, provided that $P$ has no more than $n+1$ lattice points on each facet and possesses unimodular triangulation. Then we prove Oda's question is true for any two facet unimodular polytopes whose matrix defining the facets has at most two non-zero entries in each row, and also true for any almost co-unimodular pair of reflexive polytopes. 
\end{abstract}

\maketitle

\section{Introduction}

    The Minkowski sum of two subsets $P$ and $Q$ of $\mathbb{R}^n$ is defined as:
    $$
    P+Q= \{v+u \in \mathbb{R}^n \vert v\in P, u \in Q\}.
    $$

Let $P:=\mathrm{conv}(x_1,\dots,x_s)$, $Q:=\mathrm{conv}(y_1,\dots,y_t)$, then the definition above is equivalent to 
$$
P+Q=\mathrm{conv}(\{x_i+y_j \in \mathbb{R}^n \vert i \in \{1,\dots,s\}, j \in \{1,\dots, t\}\}).
$$


In 1997, Tadao Oda \cite{Oda} asked a question about the Minkowski sum of two lattice polytopes, namely, when will the following equation holds for lattice polytopes $P$ and $Q$:
$$
(P+Q)\cap \mathbb{Z}^n = P\cap \mathbb{Z}^n + Q\cap \mathbb{Z}^n.
$$
The polytope pair $(P,Q)$ satisfying the equality above is said to be an \emph{integer decomposition property pair (IDP pair)}. It is trivially correct in dimension $1$. But in general, there are already many counterexamples in dimension $2$, e.g., $P=\mathrm{conv}((0,0),(0,1),(1,0))$ and $Q=\mathrm{conv}((0,0),(2,1),(3,1))$. The lattice point $(1,1)$ can not be decomposed into two lattice points in $P$ and $Q$. Besides, if we require that $(P,Q)$ is an IDP pair when $Q=(k-1)P$ (for any integer $k \geq 1$), the equality can be transferred to
$$
kP\cap \mathbb{Z}^n = P\cap \mathbb{Z}^n + \cdots + P\cap \mathbb{Z}^n \ (k \ \textnormal{times})
$$
which is called the \emph{integer decomposition property (IDP) of $P$ }. This is not generally true in dimension $\geq 3$, e.g., $P=\mathrm{conv}((0,0,0),(1,1,0),(1,0,1),(0,1,1))$. The point $(1,1,1) \in 2P$ can not be decomposed into two lattice points in $P$. As a well-known fact, the IDPness of a lattice polytope is corresponding to the projective normality of the related projective toric varieties (see e.g., \cite{CLS}). In Oda's paper \cite{Oda}, he raised two famous conjectures for answering the equalities above, which have captured much attention in commutative algebra, combinatorics and toric geometry. Recall that an $n$-dimensional polytope is said to be \emph{smooth} if there are exactly $n$ edges at each vertex and the primitive edge vectors form a lattice basis.

\begin{Conj} [T. Oda \cite{Oda}] 
The following two statements are true:
\begin{itemize}
    \item [1.] Any smooth polytope has the IDP.
    \item[2.]  Let $P,Q \subset \mathbb{R}^n$ be two smooth polytopes, if the normal fan of $P$ refines the normal fan of $Q$, then they form an IDP pair.
\end{itemize}
\end{Conj}

One can easily see that the second conjecture implies the first one. For the second conjecture, only a few results are known. For example, Robins \cite{SR} shows that $(P,Q)$ is an IDP pair if they are smooth and their coarse common refinement of normal fans has at most $n+3$ rays in $\mathbb{R}^n$, which extends the result of Ikeda \cite{Ike} on common refinements with at most $n+2$ rays. In addition, the IDP pair has some nice connections to the IDPness of their Cayley sums, see e.g., \cite{HOT} and \cite{Tsu}. Regarding the first conjecture, Beck et al. \cite{BHH} have shown that centrally symmetric smooth polytopes of dimension $3$ always have the IDP. Besides, it's well known that the existence of a unimodular triangualtion on a lattice polytope implies its IDPness. Though, proving IDPness and finding unimodular triangulations directly on general smooth polytopes are quite hard, people have already seen some nice results for the reflexive case. A lattice polytope is said to be \emph{reflexive} if its dual polytope is again a lattice polytope (see Subsection \ref{sub:1}).

\begin{Def}
    Let $P$ be a lattice polytope in dimension $n$ and $v$ be an arbitrary vertex of $P$. The subgroup $L$ of $\mathbb{Z}^n$ is defined as 
    $$L:= v + \sum_{x,y \in P \cap \mathbb{Z}^n} \mathbb{Z}(x-y)$$
    $P$ is said to be \emph{normal} if for any integer $k\geq 1$, the equality below holds:
    $$
    kP \cap L = P\cap L + \cdots + P\cap L \ (k \ \textnormal{times}).
    $$
\end{Def}

By the definition, one can see that $P$ is IDP if $P$ is normal and $L=\mathbb{Z}^n$. Therefore, for smooth polytopes, the IDPness is exactly the same as the normality. And these two notions are blurred in the rest of this paper. 

\begin{The}[Haase-Paffenholz, \cite{HP}] \
    \begin{itemize}
        \item All smooth reflexive polytopes in dimension at most $8$ are normal.
        \item All smooth reflexive polytopes in dimension at most $6$ have a (regualr) unimodular triangulation.
    \end{itemize}
\end{The}

Since every smooth (reflexive) polytope we know is found to be normal in higher dimension, Oda's conjecture seems to be true. Furthermore, when dealing with Oda's conjectures, one may indeed focus on the reflexive case. The following statements provide a very powerful fact, namely, every smooth polytope can be realized as a face of some smooth reflexive polytope.

\begin{The} [Wedge construction, Haase-Melnikov, \cite{HM}]
    Let $P$ be an $n$-dimensional polytope, then up to lattice equivalence, there exists a reflexive polytope $P'$ in higher dimension such that $P$ is a face of $P'$.
\end{The}

It's true that the smoothness will be kept but it's not so obvious and clear in this construction. In Santos and Kim's paper \cite{KS}, they provide an equivalent but more intuitve structure, that is,
$$W_F(P):=P \times [-1,\infty) \cap \{(x,t) \in {\mathbb{R}}^d \times \mathbb{R} : f(x)+t \le b -1 \}$$  
Details and the picture below can be found in Section $3.1$ of their paper.

\begin{figure}[ht]
    \centering
    \includegraphics[width=0.75\linewidth]{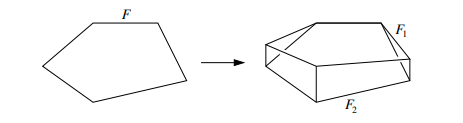}
    \caption{Wedge Structure, E. Kim and F. Santos}
    \label{fig:placeholder}
\end{figure}

With the wedge construction above, one can easily see the equivlance between the IDPness of smooth polytope and the IDPness of smooth reflexive polytopes. We give the precise description of it in the next section.

Thus, to approach Oda's conjectures, it's enough to mainly study reflexive polytopes, see Corollary \ref{Cor:2}. In this paper, we focus on several special types of reflexive polytopes. A trivial fact is that every simplicial reflexive polytope whose vertices on each facet form a lattice basis (also known as \emph{smooth Fano polytope}) has the IDP since it always possesses a unimodular triangulation. Another interesting result is that the faces of Minkowski sums can be decomposed into Minkowski sum of faces of polytopes (see \cite[Proposition 12.1]{Fuk}). But they do not necessarily form an IDP pair. If we put some restrictions on the faces, then the next statement, which is one of our main results, generalizes the trivial fact and indicates an IDP pair where unimodular triangulation of only one polytope is required. Denote $\mathcal{V}(P)$ as the set of all vertices of $P$.

\begin{The} \label{thm:3}
  Let $Q\subset \mathbb{R}^n$ be a lattice polytope containing the origin and $P \subset \mathbb{R}^n$ be a simplicial reflexive polytope of dimension $n$, any of whose facets $F$ satisfies $\vert F \cap \mathbb{Z}^n \vert \leq n+1$. If $P$ possesses a unimodular triangulation and $\mathcal{V}(Q) \subset \mathcal{V(P)} \cup \{0\}$, then $(P,Q)$ is an IDP pair.
\end{The}

After that, we prove IDP pairs for facet unimodular polytopes with at most two non-zero entries for each facet normals and almost co-unimodular pair of polytopes. A polytope is called \emph{facet unimodular} if its facet normals are row vectors of a totally unimodular matrix. A pair of polytopes $(P,Q)$ is said to be \emph{almost co-unimodular}, if their facet normals are row vectors of almost co-unimodular matrix, namely, the rows of more than two non-zero entries form a totally unimodular matrix. There are already some similar results for polytopes without requirements for their face fans or normal fans. In Danilov and Koshevoy's paper \cite{DK}, by discussing the unimodular systems, they show that any two $n$-dimensional lattice polytopes whose tanget space $\mathbb{R}(F-F):=\{c(x-y) \vert x,y \in F, c \in \mathbb{R}\}$ of each face $F$ is generated by column vectors of a full row rank unimodular $m\times n$-matrix, will form an IDP pair. Besides, Howard \cite{How} shows that any two edge unimodular polytopes, which are lattice polytopes whose edge direction vectors at each vertex are column vectors of a totally unimodular matrix, always form an IDP pair. As a corollary, the so-called pairwise facet unimodular polytopes also form IDP pairs. However, the examples mentioned by Howard shows that facet unimodular polytopes and edge unimodular polytopes are not the same objects \cite{How}. Be aware that the above results and the following are true even if $P,Q$ are not in the same dimension. The next theorem is another main result of this paper.

\begin{The}\label{thm:4}
  Let $P,Q \subset \mathbb{R}^n$ be two lattice polytopes.
    \begin{itemize}
        
        \item [1.] If $P,Q$ are facet unimodular polytopes with at most two non-zero entries in each facet normal, then $(P,Q)$ is an IDP pair.
        \item [2.] If $P,Q$ are reflexive polytopes and the pair $(P,Q)$ is almost co-unimodular, then $(P,Q)$ is an IDP pair.
    \end{itemize} 
\end{The}

About the structure of this paper, in Section $2$, we recall some basic facts about lattice geometry and fix notations.
In Section $3$, we give proofs to our main theorems, i.e., Theorems \ref{thm:3} and \ref{thm:4}.




\section{Preliminaries}
In this section, we recall some basic facts and fix notations. For details about polyhedral geometry, we recommand the book by Cox, Little and Schenck \cite{CLS}.

\subsection{dual polytopes, unimodular equivalence} \label{sub:1}
The polar set of a polytope $K$ of $\mathbb{R}^n$ is defined to be the set 
$$
K^\vee := \{u\in \mathbb{R}^n \vert \langle u,v \rangle \geq -1, \forall v\in K \}.
$$
The polar set is indeed a polytope if the origin is contained in the interior of the polytope (see e.g., \cite{Ful}), which is called the \emph{dual polytope} of $K$. A lattice polytope is \emph{reflexive} if its dual polytope is again a lattice polytope. Furthermore, if $Q$ is the dual polytope of $P$, then $P$ is exactly the dual polytope of $Q$, i.e., $P=Q^\vee = (P^\vee)^\vee$ (see \cite{CLS}).

\begin{Def}
    A linear transformation $L: \mathbb{R}^n \rightarrow \mathbb{R}^n$ is called \emph{unimodular} if it is represented by a matrix $A$ with integral entries and $\mathrm{det}A = \pm 1$. A polytope $P$ is said to be unimodular equivalent (or lattice equivalent) to $Q$ if there exists a unimodular tranformation which maps the vertices of $P$ to the vertices of $Q$.
\end{Def}
It's well known that taking dual polytopes, unimodular triangulation, IDP, IDP pairs, etc, are invariant under unimodular transformations.

\subsection{fans, subdivisions, triangulations, unimodular transformation}
\begin{Def} \label{Def:2.2}

The \emph{dual cone} of a strongly convex polyhedral cone $\sigma$ is defined to be $\sigma^\vee:= \{u\in \mathbb{R}^n \vert \langle u,v \rangle \geq 0, \forall v\in \sigma\}$. A \emph{face} $\tau$ of a strongly convex polyhedral cone $\sigma$ is the intersection of $\sigma$ with some supporting hyperplane, namely, $\tau:= \sigma \cap u^\perp= \{v \in \sigma \vert \langle v,u\rangle =0 \}$ for some $u$ in $\sigma^\vee$. A \emph{facet} of $\sigma$ is a face of codimension $1$.

A \emph{fan} $\Delta$ is a set of strongly convex polyhedral cones $\sigma$ in $\mathbb{R}^n$ such that
\begin{itemize}
    \item each face of a cone $\sigma$ in $\Delta$ is again a cone in $\Delta$, and
    \item the intersection of two cones in $\Delta$ is a face of each.
\end{itemize}
The cone which is not a proper face of any cone in $\Delta$ is called the \emph{maximal cone} in $\Delta$. 

\end{Def}

\begin{Def}
Let $P$ be any lattice polytope containing the origin in its interior. For any face $f$ of $P$, the cone $\sigma_f$ over $f$ is a cone whose rays going along the vertices of $f$ with the origin as its apex. 
\begin{itemize}
  \item The \emph{face fan} $\Delta_P$ is the fan consisting of cones over any face of $P$. 
\end{itemize}
Let $Q$ be any lattice polytope. For any vertex $v$ of $Q$, the maximal cone $\tau_v$ with respect to $v$ is defined to be a cone whose rays going along the edge direction vectors at $v$ with point $v$ as its apex. Considering the translated maximal cone $\tau_v-v$, its dual cone is denoted as $\gamma_v:={(\tau_v-v)}^\vee$. 
\begin{itemize}
 \item The \emph{normal fan} $N(Q)$ of $Q$ is the fan consisting of all dual cones $\gamma_v$ and their faces.
\end{itemize}

\end{Def}
From these definitions, we can check that the face fan $\Delta_P$ of the polytope $P$ is the normal fan $N(P^*)$ of its dual polytope $P^*$. The main theorem can be stated after recalling the notion of unimodular triangulations.

\begin{Def}
Let $P$ be a polytope of dimension $n$. A \emph{subdivision} of $P$ is a finite collection $\mathcal{S}=\{P_1,\dots,P_m\}$ of polytopes, s.t.,
\begin{itemize}
    \item the face of each $P_i$ is again in $\mathcal{S}$
    \item $P$ is the union of $P_i$
    \item for $i\neq j$, $P_i \cap P_j$ is a common face (might be empty face) of $P_i$ and $P_j$.
\end{itemize}
The maximal ($n$-dimensional) polytopes in $\mathcal{S}$ are called cells. A subdivision is said to be \emph{triangulation} if each cell is a simplex. The triangulation is called \emph{unimodular} if each simplex is unimodular, namely, unimodularly equivalent to the standard $n$-simplex $\mathrm{conv}(0,e_1,\dots,e_n)$, where $e_i$'s are standard basis vectors of $\mathbb{R}^n$. The triangulation is called \emph{centric} if each maximal cell contains the origin as its vertex.
\end{Def}

\subsection{IDPness of smooth reflexive polytopes}

By the wedge construction in Section $1$, we have the following corollary.

\begin{Cor} \
\begin{itemize}
    \item Any smooth polytope has the IDP if and only if any smooth reflexive polytope has the IDP.
    \item For any smooth polytope pair $(P,Q)$, where the normal fan of $P$ refines the one of $Q$, they form an IDP pair if and only if the reflexive ones form an IDP pair.
\end{itemize}    
\end{Cor}

\begin{proof}

    Since IDP is the same as the normality for smooth polytopes, it's obviously true by combining the constructions above and the fact that any face of a normal polytope is again normal.
    
\end{proof}

\subsection{polytopes with (almost) unimodular matrix, symmetric edge polytopes}

The hyperplane associated with a facet $F$ of lattice polytope $P$ is given by $\{x\in  \mathbb{R}^n \vert \langle n_F,x \rangle = k\}$, where $n_F$ is called the facet normal of $F$ and $k$ is an integer. The matrix whose each row represents a primitive facet normal of $P$ is denoted as $M_P$. Let $M_P$ be an $m\times n$-matrix with $n<m$. Then the matrix $M_P$ is said to be unimodular if any of its $n\times n$-submatrices has determinant $0,\pm 1$. We call a lattice polytope $P$ \emph{facet unimodular} if and only if $M_P$ is a unimodular matrix. A pair $(P,Q)$ of polytopes $P$ and $Q$ is said to be \emph{almost co-unimodular} if up to unimodular transformation, their facet normals as row vectors form a $\{-1,0,1\}$-matrix whose rows with more than two non-zero entries form a totally unimodular matrix, i.e., has determinant $-1,0,1$.

The following lemma tells a trivial fact of facet unimodular polytopes:
\begin{Lem}\label{lem:2}
    Let $M_P$ be an $m\times n$ ($n<m$) representing matrix of facet normals of lattice polytope $P \subset \mathbb{R}^n$. Then $M_P$ is unimodular if and only if there exists a unimodular transformation $L$ with representing matrix $A_L$ such that $M_P\cdot A_L$ is totally unimodular. 
\end{Lem}

\begin{proof}
 Let $F$ be a facet of $P$. We pick a unimodular transformation $L$ such that the vertices of the new facet $L(F)$ are given by standard basis vectors $e_i$ of $\mathbb{R}^n$. Then $M_P\cdot A_L = \begin{pmatrix} I \\ M'_P 
 \end{pmatrix}$, where $I$ is the identity matrix. Since $M_P$ is unimodular, then $M_P\cdot A_L$ is again unimodular. Thus, $M'_P$ is totally unimodular by \cite{Sch}.   
    
\end{proof}

Then, for explaining Corollary \ref{Cor:2}, we quickly see the notion of symmetric edge polytopes.

\begin{Def}
    Given an undirected graph $G$ on the vertex set $\{1,\dots n+1\}$ with the edge set $E$, the \emph{symmetric edge polytope} (\emph{SEP}, for short) associated to $G$ is
    $\Sigma(G) = \mathrm{conv} (\pm(e_i-e_j)\in \mathbb{R}^{n+1} \vert \ ij\in E)$. This polytope is not of full rank since all points lying in the hyperplane $\{(x_1,\dots, x_{n+1}) \in \mathbb{R}^{n+1} \vert x_1+\cdots+x_{n+1}=0\}$. By forgetting the last entry of points, $\Sigma(G)$ will be projected to $\mathbb{R}^n$ as a full dimensional polytope. In the following contents, all the SEPs are of full rank.  
\end{Def}

\section{Main Results}

In this section, we prove our main results, Theorems \ref{thm:3} and \ref{thm:4}. 

\begin{proof}[Proof of Theorem \ref{thm:3}]
   
   For proving $(P,Q)$ an IDP pair, we want to construct a finer fan associated with $P$ than $\Delta_P$. Let $\Gamma$ be a unimodular triangulation of $P$ and let $F$ be any facet of $P$, then $\Gamma_F:=\Gamma \cap F$ provides a unimodular triangualtion of $F$. Denote $\mathcal{F}(P)$ as the collection of all facets of $P$. Thus, $\partial \Gamma :=\bigcup_{F\in \mathcal{F}(P)} \Gamma_F$ is a unimodular triangulation of the boundary of $P$. Consider the centric triangulation $\Gamma(P)$ by taking the convex hulls of simplices in $\partial \Gamma$ and the origin. Then $\Gamma(P)$ is unimodular since $P$ is reflexive and $\partial \Gamma$ is unimodular.

   Let $x\in (P+Q) \cap \mathbb{Z}^n$. We would like to show that $x$ can be written as $x=y+z$, where $y\in P\cap \mathbb{Z}^n$ and $z\in Q\cap \mathbb{Z}^n$. First, we have $x\in 2P$. Assume that $F$ is a facet of $P$ such that $x\in \mathrm{cone}(F)$. With the construction of $\Gamma(P)$ above, we denote $S(F)$ as the union of simplices, i.e., $  S(F) = \bigcup _{\mu \in F \cap \partial \Gamma}\mathrm{conv}(\{0\},\mu)$, where $\mathrm{conv}(\{0\},\mu)$ is unimodular. Thus, $S(F)$ has IDP, and then, $x \in 2S(F) \subset 2P$ can be written as $x=v_1+v_2$ for $v_1,v_2\in S(F) \cap \mathbb{Z}^n$. 
   
   If $v_1 \in Q$ or $v_2 \in Q$, there is nothing to prove. So, we suppose that $v_1,v_2\in S(F) \backslash Q$. Let $y \in P$ and $z \in Q$ be two rational points satisfying $x=y+z$, then $x=y+z=v_1+v_2$.  Be aware that $P$ is set to be reflexive, so the facets have exactly lattice distance $1$ to the origin. In other words, for the facet normal $e_F$ of $F$, we have $\langle e_F, H(F) \rangle=1$, where $H(F)$ is the supporting hyperplane of $F$. By applying $e_F$ on both sides of the equation $y+z=v_1+v_2$, it shows:
\begin{itemize}
\item
$\langle e_F, z \rangle = \langle e_F , v_1+v_2-y\rangle \geq 1$ (because $v_1,v_2$ lie on $F$ and $\langle e_F, y \rangle \leq 1$ since $e_F$ defines a supporting hyperplane).
On the other hand, $z\in Q \subset P$ implies that $\langle e_F, z \rangle \leq 1$. Therefore, $\langle e_F, z \rangle =1$ and $z\in H(F)\cap P = F$. 

\item 
Similarly, by applying $e_F$ on $y$, we have $y \in F$.

\item 
Furthermore, we can let $z$ lie on a fixed cone $\mathrm{cone}(G)$ associated with some face $G$ of $Q$. To be precise, we construct a fan $\Delta$ of cones in the following way:
\begin{itemize}
    \item [1.] if the origin $0$ is contained in the relative interior of $Q$, then the cones over all facets of $Q$ partition the whole space and this is exactly the face fan $\Delta$ w.r.t. $Q$. 
    \item [2.] if $0$ lies on the boundary of $Q$, we may assume that there exists a face $K$ such that the origin $0$ is contained in the relative interior of $K$. Then for any other face $L$ of $Q$, we can still construct a cone $\mathrm{C_L}$ over the face $L$ with origin as the apex. For $K$, we may regard it as a polytope containing the origin in its relative interior. Then we again construct cones $\mathrm{C_{K'}}$ by taking cones over facets $K'$ of $K$ with origin as the apex. From the Definition \ref{Def:2.2}, one can check that the collection of $\mathrm{C_L}$ and $\mathrm{C_{K'}}$ and their intersections form a fan $\Delta$. 
\end{itemize}
Therefore, the set $P:=\bigcup_{C_{max} \in \Delta} C_{max} \cap Q$, where $C_{max}$ is the maximal cone in $\Delta$, is a subdivision of the polytope $Q$. Then for any $z\in Q$, we can indeed find a cone $\mathrm{cone}(G)$ such that $z\in \mathrm{cone}(G)$. 

\end{itemize}

If $z\in G$, then $z\in F \cap G$. Otherwise, $z \in \mathrm{cone}(G) \backslash G$, then $z\in Q \subset P$ is an interior point of $Q$, and thus, also an interior point of $P$. This contradicts to the fact that $z\in F$ is a boundary point of $P$. So, $z\in F \cap G \subset S(F)$. Now, we claim that $F \cap G$ is a face of both $F$ and $G$.

\begin{proof}[Proof of the claim]

Regarding $F\cap G$ as a polytope, let $v \in F\cap G$ be a vertex and $v\notin \mathcal{V}(F)$, then $v \notin \mathcal{V}(G)$ since $v\in \mathcal{V}(G) \subset \mathcal{V}(Q) \subset \mathcal{V}(P) \cup \{0\}$ implies that $v\in \mathcal{V}(P)\cap F = \mathcal{V}(F)$. Thus $v$ is a relative interior point of some face of $G$, namely, there exists two vertices of $G$ such that they lie on two different sides of the hyperplane $H(F)$ across $F$, which contradicts with $F$ being a facet. So, $v \in \mathcal{V}(F)$. Recalling that $P$ is simplicial, then $\mathcal{V}(F\cap G) \subset \mathcal{V}(F)$ implies that $F\cap G$ is a face of $F$. Since there do not exist lattice points lying on both sides of $H(F)$, $H(F)\cap G$ is a face of $G$. Besides, $G \subset Q \subset P$ and $F$ being a facet of $P$ implies that $F \cap \mathcal{V}(G) = H(F) \cap \mathcal{V}(G)$. By taking the convex hull of both sides of the equality, we have $F\cap G = H(F) \cap G$ is a face of $G$.

\end{proof}
 Now, we come back to our discussion about $v_1$ and $v_2$. By our assumption above, $v_1,v_2 \in F \backslash G$. Since $\vert F \cap \mathbb{Z}^n \vert \leq n+1 $, there may be two situations about $v_1,v_2$:

\begin{figure}[h]
    \centering
    \includegraphics[width=0.5\linewidth]{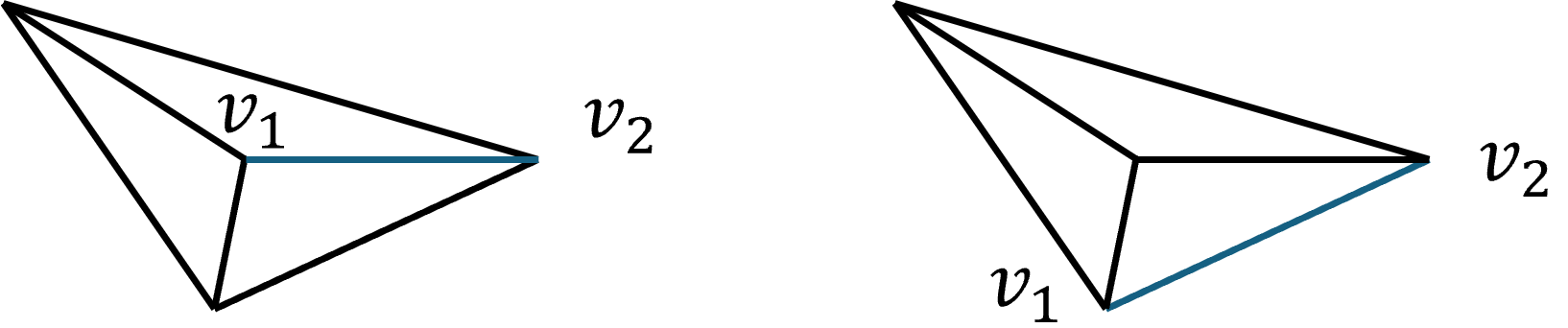}
    \caption{only two possible choices for $v_1$ and $v_2$}
    \label{fig:placeholder}
\end{figure}

\begin{itemize}
    
    \item [1.] If one of $v_1,v_2$ is the lattice point other than vertices of $F$: w.l.o.g, let $v_1$ be that lattice point, then $v_2$ will be the vertex in $F$ which is not on the face $F\cap G$. 
    Since $P$ is simplicial, we may let $e_{F\cap G}$ be a facet normal of the supporting hyperplane $H(F \cap G)$ of $P$ such that $v_2$ is the unique vertex of $F$ not on it. In other words, $v_2$ has the farthest distance to $H(F\cap G)$. Then we have an inequality $\langle e_{F\cap G},H(F \cap G) \rangle \leq a$ for some positive integer $a$. Since $z\in F \cap G$, $y \in F$ and $v_2$ has the farthest distance to the hyperplane $H(F\cap G)$ in $H(F)$, we have 
    $$
    \langle e_{F\cap G}, y+z \rangle = \langle e_{F\cap G},y \rangle + a \geq \langle e_{F\cap G},v_2 \rangle + a.
    $$
    On the other hand, $v_1 \in F$ implies
    $$
    \langle e_{F\cap G}, v_1+v_2 \rangle =\langle e_{F\cap G}, v_2 \rangle +\langle e_{F\cap G}, v_1 \rangle \leq \langle e_{F\cap G}, v_2 \rangle +a .
    $$
    Therefore, the inequalities above should be equalities, namely, $\langle e_{F\cap G},v_1\rangle =a$. Thus, $v_1 \in F\cap G \subset Q$, contradiction.
    
    \item [2.] If $v_1,v_2$ are vertices of $F$: since $P$ is simplicial, the line segment $l(v_1v_2)$ connecting $v_1$ and $v_2$ is a face of $P$. By the previous discussion, $z\in F\cap G \subset G$. If $G \cap l(v_1v_2)$ is not empty, then $v_1$ or $v_2$ lies on $G$ and thus in $Q$, contradiction.  If $z \in l(v_1v_2)$, then $z\in l(v_1v_2)\cap G = \emptyset$, contradiction. Thus, $z\notin l(v_1v_2)$. Since $y+z =v_1+v_2$, by the convexity of polytopes, $y$ lies out of $P$, contradiction.
\end{itemize}

\noindent
Therefore, $x=v_1+v_2$, where at least one of $v_1,v_2$ lies on $Q$.
   
\end{proof}

The following is a direct consequence of the previous theorem. Recall that a smooth Fano polytope is a simplicial reflexive polytope whose vertices on each facet form a lattice basis.

\begin{Cor}\label{Cor:3}
    Let $P,Q$ be two smooth Fano polytopes, where the face fan of $P$ refines the face fan of $Q$. Then $(P,Q)$ is an IDP pair, namely,
    $$ (P+Q) \cap \mathbb{Z}^n = P  \cap \mathbb{Z}^n + Q \cap \mathbb{Z}^n.$$
\end{Cor}

Finally, we will give a proof of Theorem \ref{thm:4}. By Lemma \ref{lem:2}, a facet unimodular polytope induces a totally unimodular matrix for its facet normals up to lattice equivalence. With this result, the proof of the first statement can be given by applying ceilings and floors on the decomposition points $y$ and $z$. Then, the second one can be seen as a generalization of the first one. In the following proof, we denote $M_P$ and $M_Q$ as the representing matrices of facet normals of $P$ and $Q$, resepctively. 

\begin{proof}[Proof of Theorem \ref{thm:4}]
    For the first statement:
    Let $x=y+z \in (P+Q) \cap \mathbb{Z}^n$. Denote $e_i$ as the standard basis vector of $\mathbb{R}^n$. For those rows with only one non-zero entry, we always have $\langle e_i, y \rangle =y_i \geq -a_k \ (a_k \in \mathbb{Z}) \implies \lfloor y_i \rfloor \geq -a_k$ and $\lceil y_i \rceil \geq -a_k$ (same for $z_j$). Then it's obvious that $x=y+z=\lfloor y \rfloor + \lceil z \rceil$, where  $\lfloor y \rfloor \in P$ and $\lceil z \rceil \in Q$. For those rows with two non-zero entries, let them be, e.g., $e_i-e_j$. Since by multiplying $-1$ on columns of $M_P$, we can avoid the type $e_i+e_j$ by totally unimodularity of $M_P$ (same for $M_Q$). Then $\langle y, e_i-e_j \rangle \geq -a_k \implies \lfloor y_i \rfloor - \lfloor y_j \rfloor \geq -a_k$ and $\lceil y_i \rceil - \lceil y_j \rceil \geq -a_k$ (same for $z_i$ and $z_j$). So if $e_i-e_j \in M_P$ (or $-e_i+e_j \in M_P$) and $e_i-e_j \in M_Q$ or $-e_i+e_j \in M_Q$, then $x$ can be written as $x = \lfloor y \rfloor + \lceil z \rceil$. If not, then it's the same as the case where only one non-zero entry lies on the row.

    \par
    \vspace{10pt}
    The second statement here is similar but a bit complicated:
    W.l.o.g., we may assume that $M$ is the smallest almost co-unimodular matrix containing $M_P$ and $M_Q$ as submatrices (since we can find a large almost co-unimodular matrix containing $M_P$ and $M_Q$, and reducing it to the smallest one by deleting rows or columns which will keep the totally unimodularity of the submatrices that already have it). There are several cases:
    \begin{itemize}
        \item [(1)] If each row of $M$ contains at most two non-zero entries: Then it degenerates to the same situation of the first statement.
        \item [(2)] If a row of $M$ may contain more than two non-zero entries: Since $M$ is almost co-unimodular, by deleting rows with at most two non-zero entries, then the rest part of the matrix, saying $M'$, is totally unimodular. Recalling the equivalence between (\rmnum{1}) and (\rmnum{4}) in \cite[Theorem 19.3]{Sch}, $M'$ can be transformed into a matrix by multiplying $-1$ on some columns such that the sum of all columns is a vector with entries $0,\pm 1$. Then we first apply the next method (a) in the case where each row of $M$ has at most three non-zero entries and finish the proof by induction.
        
        \item [(a)] Let $m_P, m_Q$ be two rows of $M'$ and can be seen as rows of $M_P$ and $M_Q$, respectively. If they share three non-zero entries in the same position, e.g., $i,j,k$, then the inequality of the hyperplane defined by the facet normal $m_P$ may be written as $y_i-y_j-y_k \geq -1$. In other words, $-y_i-y_j-y_k$ will never appear since the sum of coefficients is not $0, \pm 1$, which is the same for $m_Q$. W.l.o.g., let the defining inequalities of $m_P$ and $m_Q$ be $y_i-y_j-y_k \geq -1$ and $-z_i+z_j+z_k \geq -1$, respectively.
        
        It's trivially true that 
        $$y_i-y_j-y_k \geq -1 \implies \lfloor y_i \rfloor -\lfloor y_j \rfloor- \lfloor y_k \rfloor \geq -1$$ 
        And we have 
        $$y_i-y_j-y_k \geq -1 \implies \lceil y_i \rceil - \lceil y_j \rceil - \lceil y_k \rceil \geq -1$$ by the fact that $-1\leq y_t\leq 1$ for all $t$.  
        On the other hand, 
        $$-z_i+z_j+z_k \geq -1 \implies -\lceil z_i \rceil + \lceil z_j \rceil + \lceil z_k \rceil \geq -1$$ is trivially true and 
        $$-z_i+z_j+z_k \geq -1 \implies -\lfloor z_i \rfloor +\lfloor z_j \rfloor+ \lfloor z_k \rfloor \geq -1$$ by the facts that $-1\leq z_t\leq 1$ for all $t$.
        Therefore, $x=\lfloor y \rfloor + \lceil z \rceil$ since these two integer points are contained in the polytopes.

        \item [(b)] If they do not share three non-zero entries in the same position: Then it's a degenerate situation. For $m_Q$, it may have non-zero entries in different position or have less than three non-zero entries. The latter case is obvious since there is no requirement for the coefficients if a row has at most two non-zero entries by the first statement. Regarding the former one, we have already known that the sum of all entries of $m_Q$ is $\pm 1, 0$. So we can still apply the method in (2)(a). 
        
        \item [(3)] If a row of $M$ may contain at most $n$ non-zero entries: Considering the non-degenerate situation, then the key point is that we can always make sure the difference between the number of positive signs and the number of negative signs is $-1, 0$ or $1$ by totally unimodularity of $M'$. So, by induction on the upper-bound of number of non-zero entries of rows, we finish the proof via obtaining $x=\lfloor y \rfloor + \lceil z \rceil$ from $x=y+z$.
    \end{itemize}

\end{proof}

The corollary below immediately follows from the proof above:

\begin{Cor}
    Let $P,Q$ be two reflexive polytopes, if the facet normals of $P$ and $Q$ form a unimodular matrix, then $(P,Q)$ is an IDP pair.
\end{Cor}

Symmetric edge polytopes are fruitful objects and nice examples in combinatorics and toric geometry. The IDP pair property of their duals follows from the theorem above by recalling that the dual of any SEP is a facet unimodular polytope.

\begin{Cor}\label{Cor:2}
    Let $P,Q$ be duals of symmetric edge polytopes, then $(P,Q)$ is an IDP pair.
\end{Cor}

 The IDP pair property is not true in general for facet unimodular polytopes even though they are relatively nice objects and always possess unimodular triangulations.

\begin{Exa}
    For any two facet unimodular polytopes, they may not be IDP pair. Here is a counterexample in dimension $4$ and we refer to the database PolyDB \cite{PDB} for details of polytopes "F.4D.0114" and "F.4D.0038" below:
\begin{itemize}
    \item Polytope $P$ (F.4D.0114) with $8$ vertices and $6$ facet normals:
\par
Facet normals are $(0, 0, 0, -1)$, $(-1, 0, 0, 0)$, $(1, 1, 0, 1)$, $(0, -1, 0, 0)$, $(0, 0, -1, 0)$, $(0, -1, 1, 0)$
\par

    \item Polytope $Q$ (F.4D.0038) with $18$ vertices and $8$ facet normals:
\par
Facet normals are $(-1, 1, 0, 0)$, $(0, 0, 0, -1)$, $(1, -1, -1, 1)$, $(0, 0, -1, 0)$, $(-1, 0, 0, 0)$, $(0, -1, 0, 0)$, $(0, 1, 1, 0)$, $(1, -1, 0, 1)$.
\par

\end{itemize}

By using Sagemath, one can check that:
$$\vert P\cap \mathbb{Z}^n +Q \cap \mathbb{Z}^n \vert = 1192 \ \ \textnormal{and} \ \ \vert (P+Q)\cap \mathbb{Z}^n \vert = 1236.$$
There are $44$ gap points and one of them in the Minkowski sum $P+Q$ but not in $P\cap \mathbb{Z}^n + Q \cap \mathbb{Z}^n$ is e.g., $(1, 0, -1, -4)$. 
In this example, one can find that in their common matrix, rows like $(1, 1, 0, 1)$, $(1, -1, -1, 1)$ form a non-totally unimodular submatrix. So the whole one must not be totally unimodular. 

\end{Exa}


    

\section*{Acknowledgement}
I thank Mateusz Michalek and Ramen Sanyal for their suggestions on the Minkowski sums of duals of symmetric edge polytopes in FPSAC 2025. In particular, the essential idea of the proof for the first statement in Theorem \ref{thm:4} was mainly given by Michalek. Besides, I do appreciate Fransisco Sanotos for pointing out the equivalent type of wedge structure. Finally, I also want to thank Akihiro Higashitani for his guidance on this paper.

\end{document}